\documentclass[12pt,reqno]{amsart}
\usepackage{amsfonts}
\usepackage{amssymb, amsmath, latexsym}
\usepackage{lipsum} 
\usepackage{mathtools}

\usepackage{array}

\newcolumntype{M}[1]{>{\centering\arraybackslash}m{#1}}
\setbox0=\hbox{$+$}
\newdimen\plusheight
\plusheight=\ht0
\def\+{\;\lower\plusheight\hbox{$+$}\;}

\setbox0=\hbox{$-$}
\newdimen\minusheight
\minusheight=\ht0
\def\-{\;\lower\minusheight\hbox{$-$}\;}

\setbox0=\hbox{$\cdots$}
\newdimen\cdotsheight
\cdotsheight=\plusheight
\def\cds{\lower\cdotsheight\hbox{$\cdots$}}

\newcommand\cplus{\mathbin{\raisebox{-\height}{$+$}}}
\newcommand\contdots{\raisebox{-\height}{$\vphantom{+}\dotsm$}}

\usepackage{mleftright}

\renewcommand{\(}{\left\(}
\renewcommand{\)}{\right\)}
\renewcommand{\[}{\left[}

\numberwithin{equation}{section}
 \theoremstyle{plain}
\newtheorem{theorem}{Theorem}[section]
\newtheorem{lemma}[theorem]{Lemma}
\newtheorem{conjecture}[theorem]{Conjecture}
\newtheorem{corollary}[theorem]{Corollary}

\newtheorem{remark}[theorem]{Remark}

\makeatletter
\def\@bignumber#1#2{%
  \ifx#2\end
    #1\let\next\@gobble
  \else
    #1\hspace{0pt plus 1pt}\let\next\@bignumber
  \fi
  \next#2}
\newcommand{\bignumber}[1]{\@bignumber#1\end}
\makeatother

\begin{document}
\allowdisplaybreaks
\title[Results on some partition functions arising from certain relations involving the Rogers-Ramanujan continued fractions] {Results on some partition functions arising from certain relations involving the Rogers-Ramanujan continued fractions}

\author{Nayandeep Deka Baruah}
\address{Department of Mathematical Sciences, Tezpur University, Assam, India, Pin-784028}
\email{nayan@tezu.ernet.in}

\author{Nilufar Mana Begum}
\address{Department of Mathematical Sciences, Tezpur University, Assam, India, Pin-784028}
\email{nilufar@tezu.ernet.in}
\author{Hirakjyoti Das}
\address{Department of Mathematical Sciences, Tezpur University, Assam, India, Pin-784028}
\email{hdas@tezu.ernet.in}


\begin{center}
{\textbf{Results on some partition functions arising from certain relations involving the Rogers-Ramanujan continued fractions}}\\[5mm]
{\footnotesize  Nayandeep Deka Baruah, Nilufar Mana Begum, and Hirakjyoti Das}\\[3mm]
\end{center}


\noindent{\footnotesize{\bf Abstract.}
Relations involving the Rogers-Ramanujan continued fractions $R(q)$, $R(q^3)$, and $R(q^4)$ are used to find new generating functions and congruences  modulo 5 and 25 for $3$-core, 4-core, $4$-regular, and colored partition functions.

\vskip 3mm
\noindent {\footnotesize 2010 Mathematical Reviews Classification
Numbers: Primary 11P83; Secondary 05A15, 05A17}.}

\section{\textbf{Introduction}}
\label{intro}

A partition of a positive integer $n$ is a finite non-increasing sequence $\lambda=(\lambda_1, \lambda_2,\ldots,\lambda_k)$ of positive integers $\lambda_1, \lambda_2,\ldots,\lambda_k$, called \emph{parts} of $\lambda$, such that
$$\sum_{j=1}^k\lambda_j=n.$$

 Let $Q(n)$ denote the number of partitions of $n$ into distinct parts (equivalently, by Euler's theorem, into odd parts). In \cite{bb-distinct} Baruah and Begum found the exact generating functions for $Q(5n+1)$, $Q(25n+1)$, $Q(125n+26)$, and a few congruences modulo 5 and 25. They used  Ramanujan's theta function
identities and some identities for the Rogers-Ramanujan continued fraction. In particular, certain relations between the continued fractions $R(q)$ and $R(q^2)$ were employed to establish their results, where $R(q)=q^{1/5}/\mathcal{R}(q)$ with $\mathcal{R}(q)$ being the famous Rogers-Ramanujan continued fraction  usually given by
\begin{align*}
   \mathcal{R}(q):= \dfrac{q^{1/5}}{1}\cplus \frac{q}{1}\cplus\frac{q^2}{1}\cplus  \frac{q^3}{1}\cplus\contdots=q^{1/5}\dfrac{(q;q^5)_\infty(q^4;q^5)_\infty}
   {(q^2;q^5)_\infty(q^3;q^5)_\infty},
\end{align*}
where for any complex number $a$ and $|q|<1$, the standard $q$-product $(a;q)_\infty$ is defined by
$$(a;q)_\infty:=\prod_{k=0}^{\infty}(1-aq^{k}),$$
and the product representation of $\mathcal{R}(q)$ is due to both Rogers and Ramanujan (See \cite[pp. 158--160]{Spirit}). The technique in \cite{bb-distinct} was further used effectively in \cite{bb-appell, bb-mock, chern, ChernTang}.

In this work, we explore relations involving the Rogers-Ramanujan continued fraction $R(q)$ with those of $R(q^3)$ and $R(q^4)$ to deduce some new generating functions and congruences modulo 5 and 25 for certain partition functions mentioned in the following.

The  Ferrers-Young diagram of a partition $\lambda=(\lambda_1, \lambda_2,\ldots,\lambda_k)$ is an array of left-aligned nodes with $\lambda_i$ nodes in the $i$-th row.  If $\lambda^{\prime}_j$ denotes the number of nodes in column $j$, then the hook number of the node $(i, j)$ is defined by $H(i, j):= \lambda_i + \lambda^{\prime}_j -i-j +1$. A partition of $n$ is called a $t$-core of $n$ if none of the hook numbers is a multiple of $t$. For example, the Ferrers-Young diagram of the partition $\lambda=(5,2,1)$ is given by
$$
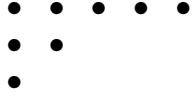
\begin{array}{ccccc}
  \bullet & \bullet & \bullet & \bullet & \bullet \\
  \bullet & \bullet &  &  & \\
  \bullet&  &  &  &
\end{array}$$
The nodes $(1,1), (1,2), (1,3), (1,4), (1,5), (2,1), (2,2)$, and $(3,1)$ have hook numbers 7, 5, 3, 2, 1, 3, 1, and $1$, respectively. Since none of these is a multiple of 4, so $\lambda$ is a $4$-core. Obviously, it is a $t$-core for $t\geq8$. Let $a_{t}(n)$ denote the number of partitions of $n$ that are $t$-cores. It is well-known that  the generating function for $a_{t}(n)$ is given by
\begin{align*}
   \sum_{n=0}^{\infty}a_{t}(n)q^{n}=\dfrac{E_t^{t}}{E_1},
\end{align*}
where here and throughout the sequel, for a positive integer $n$,
\begin{align*}
    E_n:=(q^n;q^n)_\infty.
\end{align*}

Next, a partition of a positive integer $n$ is said to be $\ell$-regular if none of its parts is divisible by $\ell$. If $b_{\ell}(n)$ denotes the number of $\ell$-regular partitions of $n$ with $b_{\ell}(0)=1$, then the generating function for $b_\ell(n)$ is given by
$$\sum_{n\geq0}b_{\ell}(n)q^n=\dfrac{E_{\ell}}{E_1}.$$

Finally, for an integer $r>1$, let $p_{[1^kr^k]}(n)$ denote the number of partitions of $n$ in which multiples of $r$ appear in $k+1$ colors and the rest of the parts in $k$ colors. The generating function for $p_{[1^kr^k]}(n)$ is given by
\begin{align*}
    \sum_{n=0}^{\infty}p_{[1^kr^k]}(n)q^n&=\dfrac{1}{E_1^kE_r^k}.
\end{align*}

We are now in a position to state our results. The results  from Theorem \ref{Theo1}  through Theorem \ref{TheoCongruence} are obtained by using identities involving $R(q)$ and $R(q^3)$.

\begin{theorem} \label{Theo1} For any integer $n\geq0$, we have
\begin{align}
   \label{3coregen} \sum_{n=0}^{\infty}a_{3}(5n+3)q^{n}&=q\dfrac{E_{15}^3}{E_5}.
\end{align}
\end{theorem}

With the aid of \eqref{3coregen}, we find the following corollary.
\begin{corollary}\label{coro1}For any integer $n\geq0$, $k>0$, and $r=0,2,3\;and\; 4$, we have
\begin{align}
   \label{3coreR1} a_3(25n+5r+3)&=0\\ \intertext{and}
    \label{3coreR2} a_3\bigg(5^{2k}n+\dfrac{5^{2k}-1}{3}\bigg)&=a_3(n).
\end{align}
\end{corollary}
\noindent We note that  \eqref{3coreR2} is only a special case of a more general result by Hirschhorn and Sellers \cite[Corollary 8]{Hirsch.Sell}.

\begin{theorem}\label{Theo2}For any integer $n\geq0$, we have
\begin{align}
\label{colorgen1}  \sum_{n=0}^{\infty}p_{[1^13^1]}(5n+1)q^n&=\dfrac{E_5^5}{E_1^6E_{15}}+10q\dfrac{E_5^{10}}{E_1^7E_3^5}
  +q^2\dfrac{E_{15}^5}{E_3^6E_5}+45q^3\dfrac{E_5^5E_{15}^5}{E_1^6E_3^6}-90q^5\dfrac{E_{15}^{10}}{E_1^5E_3^7},
\end{align}

\begin{align}
\label{colorgen2}  \sum_{n=0}^{\infty}p_{[1^23^2]}(5n+2)q^n&=5\bigg(\dfrac{E_{5}^{10}}{E_{1}^{12}E_{15}^{2}}+20q\dfrac{E_{5}^{15}}{E_{1}^{13}E_{3}^{5}E_{15}^{}}+q^2\bigg(80\dfrac{E_{5}^{20}}{E_{1}^{14}E_{3}^{10}}+12\dfrac{E_{5}^{4}E_{15}^{4}}{E_{1}^{6}E_{3}^{6}}\bigg)\nonumber \\
  &\quad+ q^3\bigg(306\dfrac{E_{5}^{10}E_{15}^{4}}{E_{1}^{12}E_{3}^{6}}-36\dfrac{E_{5}^{9}E_{15}^{5}}{E_{1}^{7}E_{3}^{11}}\bigg)+q^4\bigg(\dfrac{E_{15}^{10}}{E_{1}^{12}E_{3}^{2}}+540\dfrac{E_{5}^{15}E_{15}^{5}}{E_{1}^{13}E_{3}^{11}}\bigg)\nonumber \\
  &\quad+q^5\bigg(306\dfrac{E_{5}^{4}E_{15}^{10}}{E_{1}^{6}E_{3}^{12}}+324\dfrac{E_{5}^{5}E_{15}^{9}}{E_{1}^{11}E_{3}^{7}}\bigg)+4745q^6\dfrac{E_{5}^{10}E_{15}^{10}}{E_{1}^{12}E_{3}^{12}}\nonumber \\
  &\quad-180q^7\dfrac{E_{15}^{15}}{E_{1}^{5}E_{3}^{13}E_{5}^{}}-4860q^8\dfrac{E_{5}^{5}E_{15}^{5}}{E_{1}^{11}E_{3}^{13}}+6480q^{10}\dfrac{E_{15}^{20}}{E_{1}^{10}E_{3}^{14}}\bigg)
\end{align}
and
\begin{align}
 \label{colorgen3} &\sum_{n=0}^{\infty}p_{[1^33^3]}(5n+3)q^n\nonumber\\
 &=
  25\dfrac{ E_5^{15}}{E_1^{18} E_{15}^3}
 +q \bigg(216\dfrac{ E_5^{35}}{E_1^{22} E_3^{14} E_{15}^5}+234\dfrac{ E_5^{20}}{E_1^{19} E_3^5 E_{15}^2}+234\dfrac{ E_5^{19}}{E_1^{14} E_3^{10} E_{15}}+81\dfrac{ E_{15}^3 E_5^3}{E_1^6 E_3^6}\bigg)
  \nonumber \\
  &\quad+q^2\bigg(2976\dfrac{ E_5^{25}}{E_1^{20} E_3^{10} E_{15}}+2049\dfrac{ E_5^{24}}{E_1^{15} E_3^{15}}+27\dfrac{ E_{15}^2 E_5^{10}}{E_1^{17} E_3}+3528\dfrac{ E_{15}^3 E_5^9}{E_1^{12} E_3^6}\bigg)\nonumber \\
  &\quad+q^3 \bigg(13224\dfrac{ E_5^{30}}{E_1^{21} E_3^{15}}+38858\dfrac{ E_{15}^3 E_5^{15}}{E_1^{18} E_3^6}+6075\dfrac{ E_{15}^4 E_5^{14}}{E_1^{13} E_3^{11}}+134\dfrac{ E_{15}^5 E_5^{13}}{E_1^8 E_3^{16}}\bigg)
  \nonumber \\
  &\quad+q^4 \bigg(109101\dfrac{ E_{15}^4 E_5^{20}}{E_1^{19} E_3^{11}}+22226\dfrac{ E_{15}^5 E_5^{19}}{E_1^{14} E_3^{16}}+3528\dfrac{ E_{15}^9 E_5^3}{E_1^6 E_3^{12}}-3\dfrac{ E_{15}^{10} E_5^2}{E_1 E_3^{17}}\bigg)\nonumber \\
  &\quad+q^5 \bigg(270086\dfrac{ E_{15}^5 E_5^{25}}{E_1^{20} E_3^{16}}+18018 \dfrac{E_{15}^8 E_5^{10}}{E_1^{17} E_3^7}+138555\dfrac{ E_{15}^9 E_5^9}{E_1^{12} E_3^{12}}-2002\dfrac{ E_{15}^{10} E_5^8}{E_1^7 E_3^{17}}\bigg)\nonumber \\
  &\quad+q^6 \bigg(1270766\dfrac{ E_{15}^9 E_5^{15}}{E_1^{18} E_3^{12}}-55171\dfrac{ E_{15}^{10} E_5^{14}}{E_1^{13} E_3^{17}}+25\dfrac{ E_{15}^{15}}{E_3^{18} E_5^3}\bigg)\nonumber \\
  &\quad +q^7 \bigg(441524\dfrac{ E_{15}^{10} E_5^{20}}{E_1^{19} E_3^{17}}+10854\dfrac{ E_{15}^{13} E_5^5}{E_1^{16} E_3^8}-54675\dfrac{ E_{15}^{14} E_5^4}{E_1^{11} E_3^{13}}+38858\dfrac{ E_{15}^{15} E_5^3}{E_1^6 E_3^{18}}\bigg)\nonumber \\
  &\quad+q^8 \bigg(1270766\dfrac{ E_5^9 E_{15}^{15}}{E_1^{12} E_3^{18}}+496539 \dfrac{E_5^{10} E_{15}^{14}}{E_1^{17} E_3^{13}}\bigg)\nonumber \\
  &\quad+q^9 \bigg(12104417\dfrac{ E_5^{15} E_{15}^{15}}{E_1^{18} E_3^{18}}-2106\dfrac{ E_{15}^{20}}{E_1^5 E_3^{19} E_5^2}+18954\dfrac{ E_{15}^{19}}{E_1^{10} E_3^{14} E_5}\bigg)\nonumber \\
  &\quad+q^{10} \bigg(1800306\dfrac{ E_5^5 E_{15}^{19}}{E_1^{16} E_3^{14}}-981909\dfrac{ E_5^4 E_{15}^{20}}{E_1^{11} E_3^{19}}\bigg)-3973716q^{11}\dfrac{  E_5^{10} E_{15}^{20}}{E_1^{17} E_3^{19}}\nonumber \\
  &\quad+q^{12} \bigg(241056\dfrac{ E_{15}^{25}}{E_1^{10} E_3^{20} E_5}-1493721\dfrac{ E_{15}^{24}}{E_1^{15} E_3^{15}}\bigg)+21876966q^{13}\dfrac{  E_5^5 E_{15}^{25}}{E_1^{16} E_3^{20}}\nonumber \\
  &\quad-9640296 q^{15}\dfrac{ E_{15}^{30}}{E_1^{15} E_3^{21}}+1417176 q^{17}\dfrac{ E_{15}^{35}}{E_1^{14} E_3^{22} E_5^5}.
\end{align}
\end{theorem}
We derive the following congruences from \eqref{colorgen1}.
\begin{corollary}\label{coro2}For any integer $n\geq0$, we have
\begin{align}
  \label{colgen1R1}  p_{[1^13^1]}(25n+21)&\equiv 0\; (\textup{mod}\: 5)\\ \intertext{and}
    \label{colgen1R2}  p_{[1^13^1]}(625n+521)&\equiv 0\; (\textup{mod}\: 5^2).
\end{align}
\end{corollary}
\noindent Congruence \eqref{colgen1R1} was  found earlier by Ahmed, Baruah, and Dastidar \cite{AhmedBaruahDastidar} whereas \eqref{colgen1R2} seems to be new. Computational evidences indicate that there might exist congruences modulo higher powers of 5 similar to \eqref{colgen1R1} and \eqref{colgen1R2}.  To that end, we pose the following conjecture.
\begin{conjecture} For any integer $n\geq0$ and $k>0$, we have
\begin{align*}
p_{[1^13^1]}\bigg(5^{2k}n+\dfrac{5^{2k+1}+1}{6}\bigg)&\equiv 0\;(\textup{mod}\:5^k).
\end{align*}
\end{conjecture}
The above infinite family of congruences is analogous to the one  for $p_{[1^12^1]}$ discovered independently  by Chan and Toh \cite{ChanandToh} and Xiong \cite{Xiong}.

We derive the next corollary from \eqref{colorgen2}.
\begin{corollary}\label{coro3}For any integer $n\geq0$ and $k>0$, we have
\begin{align}
   \label{Colgen2R1} p_{[1^23^2]}(5n+2)&\equiv 0\; (\textup{mod}\: 5)\\\intertext{and}
    \label{Colgen2R2} p_{[1^23^2]}\bigg(5^{2k}n+\dfrac{2\times5^{2k}+1}{3}\bigg) &\equiv p_{[1^23^2]}(25n+17)\; (\textup{mod}\: 5^2).
\end{align}
\end{corollary}

\begin{remark}
It is proved in Section \ref{TheoAndCoro} that
\begin{align}
    \label{Colgen3R1}p_{[1^33^3]}(5n+3)\equiv0~(\textup{mod}~5).
\end{align}
We could not transform  the generating function \eqref{colorgen3} of $p_{[1^33^3]}(5n+3)$  effectively to a form similar to \eqref{colorgen2} which could have  immediately implied the above congruence.
\end{remark}


Next, Zhang and Shi \cite{zhang} studied  the sixth order mock theta function
$\beta(q)$, defined by
\begin{align*}
  \beta(q):=\sum_{n=0}^{\infty}\dfrac{q^{3n^2+3n+1}}{(q;q^3)_{n+1}(q^2;q^3)_{n+1}}=:\sum_{n=0}^{\infty}p_{\beta}(n)q^n.
\end{align*}
They proved that
  \begin{align*}
 \sum_{n=0}^\infty p_{\beta}(3n+1)q^n&=\dfrac{E_3^3}{E_1^2}\\\intertext{and}
\sum_{n=0}^\infty p_{\beta}(9n+5)q^n&=3\dfrac{E_3^6}{E_1^5}.
  \end{align*}
They also found some congruences for $p_{\beta}(n)$ modulo 3, 5, and 7. In particular, they proved the following three congruences by using elementary techniques.
\begin{theorem}\label{TheoCongruence}For any integer $n\geq0$, we have
\begin{align}
\label{15n7}p_{\beta}(15n+7)&\equiv0~(\textup{mod~5}),\\
\label{45n23}p_{\beta}(45n+23)&\equiv 0~(\textup{mod~15})\\\intertext{and}
\label{45n41}p_{\beta}(45n+41)&\equiv0~(\textup{mod~15}).
\end{align}
\end{theorem}
In this paper, we present alternative proofs of the above congruences by using relations between $R(q)$ and $R(q^3)$.

Finally, we present the following two new results on $b_4(n)$ and $a_4(n)$ that are obtained by  using identities involving $R(q)$ and $R(q^4)$.

\begin{theorem}\label{Theo3}
For any integer $n\geq0$, we have
\begin{align}
\label{ell1}\sum_{n=0}^\infty b_4(5n+3)q^n&=3\dfrac{E_2^2E_{10}^6}{E_1^5E_4E_{20}^2}+q\dfrac{E_2^4E_5^5E_{20}^3}{E_1^6E_4^2E_{10}^4}+4q^2\dfrac{E_2^3E_{10}E_{20}^3}{E_1^5E_4^2}.
\end{align}
\end{theorem}

\begin{theorem}\label{Theo4}For any integer $n\geq0$, we have
\begin{align}
\label{4coregen}\sum_{n=0}^{\infty}a_{4}(5n)q^{n}&=\dfrac{E_4^4E_{10}^{40}}{E_1^2E_2^8E_5^{15}E_{20}^{16}}-q\Bigg(
3\dfrac{E_4^2E_{10}^{15}}{E_1^5E_2^3E_{20}^{6}}
-4\dfrac{E_4^3E_{10}^{30}}{E_1^3E_2^6E_5^{10}E_{20}^{11}}\Bigg)
\nonumber\\
&\quad-q^2\Bigg(12\dfrac{E_4^2E_{10}^{20}}{E_1^4E_2^4E_5^{5}E_{20}^{6}}+20\dfrac{E_4E_5^{5}E_{10}^{5}}{E_1^6E_2E_{20}}-24\dfrac{E_4^3E_{10}^{35}3}{E_1^2E_2^7E_5^{15}E_{20}^{11}}\Bigg)
\nonumber\\
&\quad-q^3\Bigg(27\dfrac{E_2E_5^{10}E_{20}^{4}}{E_1^7}+60\dfrac{E_4E_{10}^{10}}{E_1^5E_2^2E_{20}}-196\dfrac{E_4^2E_{10}^{25}}{E_1^3E_2^5E_5^{10}E_{20}^{6}}\Bigg)\nonumber\\
&\quad-q^4\Bigg(83\dfrac{E_5^{5}E_{20}^{4}}{E_1^6}-456\dfrac{E_4E_{10}^{15}}{E_1^4E_2^3E_5^{5}E_{20}}\Bigg)
+q^5\Bigg(296\dfrac{E_{10}^{5}E_{20}^{4}}{E_1^5E_2}+96\dfrac{E_4E_{10}^{20}}{E_1^3E_2^4E_5^{10}E_{20}}\Bigg)\nonumber\\
&\quad+q^6\Bigg(128\dfrac{E_2E_5^{5}E_{20}^{9}}{E_1^6E_4E_{10}^{5}}
+592\dfrac{E_{10}^{10}E_{20}^{4}}{E_1^4E_2^2E_5^{5}}\Bigg)
+512q^7\dfrac{E_{20}^{9}}{E_1^5E_4}.
\end{align}
\end{theorem}

This work is organized as follows. In Section \ref{prelims}, we present some preliminary lemmas. In Section \ref{TheoAndCoro}, we prove our results in Theorem \ref{Theo1}  -- Theorem \ref{TheoCongruence}. In the final section we prove Theorem \ref{Theo3} and Theorem  \ref{Theo4}.

\section{\textbf{Preliminary lemmas}}\label{prelims}
The first lemma comprises of the well-known 5-dissections of $E_1$ and $1/E_1$.
\begin{lemma} \label{Dissect} We have
\begin{align}
   \label{E_1} E_1&=E_{25}\bigg(R(q^5)-q-\dfrac{q^2}{R(q^5)}\bigg)\\\intertext{and}
  \label{1byE_1}  \dfrac{1}{E_1}&=\dfrac{E_{25}^5}{E_{5}^6}\bigg(R(q^5)^4 + q R(q^5)^3+ 2 q^2 R(q^5)^2 + 3 q^3 R(q^5) + 5 q^4 - 3 \dfrac{q^5}{R(q^5)} \nonumber\\
  &\quad+
 2 \dfrac{q^6}{R(q^5)^2} - \dfrac{q^7}{R(q^5)^3} + \dfrac{q^8}{R(q^5)^4}\bigg).
\end{align}
\end{lemma}

\begin{proof}
See \cite[Chapter 7, pp. 161--165]{Spirit}.
\end{proof}

In the next lemma, we present two useful relations among $R(q)$, $R(q^2)$, and  $E_n$.
\begin{lemma}\label{LemmaA}
We have
\begin{align}
 \label{A}A&=A(q):= R(q)^5-\dfrac{q^2}{R(q)^5}=11q+\dfrac{E_1^6}{E_5^6}\\\intertext{and}
 B&=B(q):=\dfrac{R(q)^2}{R(q^2)}-\dfrac{R(q^2)}{R(q)^2}=4q\dfrac{E_1E_{10}^5}{E_2E_5^5}\notag.
\end{align}
\end{lemma}
\begin{proof} See \cite[Chapter 7, p. 164]{Spirit} and \cite[Lemma 2.2.1]{bb-distinct}.
\end{proof}

Some relations among $R(q)$, $R(q^3)$, and  $E_n$ are stated in the following lemma.
\begin{lemma}\label{LemmaC}We have
\begin{align}
  \label{C_1}C_1&:= \dfrac{R(q)^3}{R(q^3)}+\dfrac{R(q^3)}{R(q)^3}=2+9q^2\dfrac{E_1E_{15}^5}{E_3E_5^5},\\
\label{C_2}C_2&:= R(q) R(q^3)^3+\dfrac{q^4}{R(q) R(q^3)^3}=\dfrac{E_3E_5^5}{E_1E_{15}^5}-2q^2\\ \intertext{and}
 \label{C_3}C_3&:=R(q)^2 R(q^3)- \dfrac{R(q^3)^2}{R(q)}+q^2\dfrac{R(q)}{R(q^3)^2}-\dfrac{q^2}{R(q)^2 R(q^3)} =3q.
 \end{align}
 \end{lemma}
 \begin{proof}
 See \cite[Theorem 5.1]{Gugg2} and \cite[p. 194]{AhmedBaruahDastidar}.
 \end{proof}

Our next lemma provides a relation among $R(q)$, $R(q^4)$, and  $E_n$.

 \begin{lemma}\label{LemmaD}We have
\begin{align}
   \label{D_1}D_1&:=R(q)R(q^4)+\dfrac{q^2}{R(q)R(q^4)}=2q+\dfrac{E_1E_4E_{10}^{10}}{E_2^2E_5^5E_{20}^5}.
  \end{align}
  \end{lemma}

  \begin{proof}
From \cite[Theorem 3.3(iii)]{Gugg1}, we have
\begin{align}
 \label{psieq1} R(q)R(q^4)+\dfrac{q^2}{R(q)R(q^4)}=3q+\dfrac{\psi^2(-q)}{\psi^2(-q^5)},
\end{align}
where
\begin{align*}
\psi(q)&=\sum_{j=0}^\infty q^{j(j+1)/2}.
\end{align*}
From \cite[p. 6 and p. 11]{Spirit}, we have
\begin{align}
\label{SiFunction}\psi(q)&=\dfrac{E_2^2}{E_1}.
\end{align}

Now, from \cite[Chapter 34, p. 313]{HirschPower}, we recall that
\begin{align*}
\psi^2(q)-q\psi^2(q^5)&=\dfrac{E_2E_5^3}{E_1E_{10}}.
\end{align*}
Replacing $q$ by $-q$ in the above, we have
\begin{align}
 \label{phichieq1}\psi^2(-q)+q\psi^2(-q^5)&=\dfrac{E_2E_{-5}^3}{E_{-1}E_{10}},
\end{align}
where for a positive odd integer $n$, $E_{-n}:=(-q^n;-q^n)_\infty$. By elementary $q$-product manipulation, it follows that
\begin{align}
 \label{NegativeReplacement} E_{-n}=\dfrac{E_{2n}^3}{E_nE_{4n}}.
\end{align}
Dividing \eqref{phichieq1} by $\psi^2(-q^5)$ and then employing \eqref{SiFunction}
 and \eqref{NegativeReplacement}, we obtain
 \begin{align}
 \label{lastfrac}\dfrac{\psi^2(-q)}{\psi^2(-q^5)}+q&=\dfrac{E_1E_4E_{10}^{10}}{E_2^2E_{5}^5E_{20}^5}.
\end{align}
Identity \eqref{D_1} now follows readily from
\eqref{psieq1} and \eqref{lastfrac}.
\end{proof}

Our final lemma of this section states a relation among $R(q)$, $R(q^2)$, $R(q^4)$, and  $E_n$.
  \begin{lemma}\label{LemmaF}We have
  \begin{align*}
F&:=\dfrac{R(q)^2R(q^2)}{R(q^4)}+\dfrac{R(q^4)}{R(q)^2R(q^2)}=2+4q^2\dfrac{E_2E_{20}^5}{E_4E_{10}^5}.
\end{align*}
\end{lemma}

\begin{proof}See \cite[Theorem 3.6(ii), Lemma 1.1, p. 185]{Gugg1}.
\end{proof}

We end this section by defining an extraction operator. For a power series $\displaystyle{\sum_{n=0}^{\infty}\mathcal{A}(n)q^n}$ and $r=0,1,2,3\; \textup{and}\;4$, we define the operator $[q^{5n+r}]$ by
\begin{align*}
    [q^{5n+r}]\Bigg\{\displaystyle{\sum_{n=0}^{\infty}\mathcal{A}(n)q^n}\Bigg\}&=\displaystyle{\sum_{n=0}^{\infty}\mathcal{A}(5n+r)q^n}.
\end{align*}

\section{\textbf{Proofs of \eqref{3coregen}--\eqref{45n41} using identities satisfied by $R(q)$ and $R(q^3)$ }}\label{TheoAndCoro}


\begin{proof}[Proof of Theorem \ref{Theo1}]
We have
\begin{align*}
\sum_{n=0}^{\infty}a_{3}(n)q^{n}=\dfrac{E_{3}^{3}}{E_1}.
\end{align*}
Employing Lemma \ref{Dissect} in the above and then applying $[q^{5n+3}]$, we find that
\begin{align}
\label{3coregeneve1}
\sum_{n=0}^{\infty}a_{3}(5n+3)q^{n}&=\dfrac{E_5^5E_{15}^5}{E_1^6}\bigg(
3\left(R(q)R(q^3)^3+\dfrac{q^4}{R(q)R(q^3)^3}\right)+q\left(\dfrac{R(q^3)^3}{R(q)^4}-q^2\dfrac{R(q)^4}{R(q^3)^3}\right)\notag\\
&\quad-3\left(R(q)^4R(q^3)^2+\dfrac{q^4}{R(q)^4 R(q^3)^2}-3q\left(\dfrac{R(q^3)^2}{R(q)}-q^2\dfrac{R(q)}{R(q^3)^2}\right)\right)\notag\\
&\quad+25q^2\bigg).
\end{align}

Now, using Lemma \ref{LemmaA} and Lemma \ref{LemmaC}, we have
\begin{align}
   \label{C_4} C_4&:=\dfrac{R(q^3)^3}{R(q)^4}-q^2\dfrac{R(q)^4}{R(q^3)^3}=A-C_3(C_1+1)=\dfrac{E_1^6}{E_5^6}+2 q-27 q^3\dfrac{E_{15}^5 E_1 }{E_3 E_5^5}\\\intertext{and}
&R(q)^4R(q^3)^2+\dfrac{q^4}{R(q)^4R(q^3)^2}=C_2-(C_1-2)q^2+C_3\bigg(R(q)^2R(q^3)-\dfrac{q^2}{R(q)^2R(q^3)}\bigg).\notag\end{align}
Subtracting $3q\big(R(q^3)^2/R(q)-q^2 R(q)/R(q^3)^2\big)$ from both sides of the last identity and then employing Lemma \ref{LemmaC}, we obtain
\begin{align*}
 &R(q)^4R(q^3)^2+\dfrac{q^4}{R(q)^4R(q^3)^2}-3q\left(\dfrac{R(q^3)^2}{R(q)}-q^2\dfrac{R(q)}{R(q^3)^2}\right)=\dfrac{E_3E_5^5}{E_1E_{15}^5}+7q^2-9q^4\dfrac{E_1E_{15}^5}{E_3E_{5}^5}.
\end{align*}
Employing the above identity, \eqref{C_2}, and \eqref{C_4} in \eqref{3coregeneve1}, we arrive at \eqref{3coregen}.
\end{proof}
\begin{proof}[Proof of Corollary \ref{coro1}]
The identity \eqref{3coregen} can be written as
\begin{align*}
    \sum_{n=0}^{\infty}a_{3}(5n+3)q^{n}&=\sum_{n=0}^{\infty}a_{3}(n)q^{5n+1}.
\end{align*}
Equating the coefficients of $q^{5n+r}$  for $r=0,2,3$ and $4$ from both sides of the above, we arrive at \eqref{3coreR1}. On the other hand, equating the coefficients of $q^{5n+1}$, we have
\begin{align*}
       a_3(25n+8)=a_3(n)
\end{align*}
from which \eqref{3coreR2} follows by induction.
\end{proof}

\begin{proof}[Proof of  Theorem \ref{Theo2}]
First, we prove \eqref{colorgen1}. We have
\begin{align*}
\sum_{n=0}^{\infty}p_{[1^13^1]}(n)q^n=\dfrac{1}{E_1E_3}.
\end{align*}
Employing \eqref{1byE_1} in the above and  then applying $[q^{5n+1}]$, we find that
\begin{align}
\label{colgen1eve1}&\sum_{n=0}^{\infty}p_{[1^13^1]}(5n+1)q^n\notag\\
&=\dfrac{E_5^5E_{15}^5}{E_1^6E_3^6}\bigg(25q^3+6q^2\bigg(R(q)^2R(q^3)+q^2\dfrac{R(q)}{R(q^3)^2}-\dfrac{R(q^3)^2}{R(q)}-\dfrac{q^2}{R(q)^2R(q^3)}\bigg)\notag\\
&\quad+2q\bigg(R(q)^4R(q^3)^2+q^4\dfrac{R(q)^2}{R(q^3)^4}+\dfrac{R(q^3)^4}{R(q)^2}+\dfrac{q^4}{R(q)^4R(q^3)^2}\bigg)\notag\\
&\quad-3q\bigg(q^2\dfrac{R(q)^3}{R(q^3)}-R(q)R(q^3)^3-\dfrac{q^4}{R(q)R(q^3)^3}+q^2\dfrac{R(q^3)}{R(q)^3}\bigg)\notag\\
&\quad+\bigg(R(q)^3R(q^3)^4-\dfrac{q^6}{R(q)^3R(q^3)^4}+\dfrac{R(q^3)^3}{R(q)^4}-q^4\dfrac{R(q)^4}{R(q^3)^3}\bigg)\bigg).
\end{align}

Now, with the aid of Lemma \ref{LemmaA} and Lemma \ref{LemmaC}, we have
\begin{align}
    \label{C_5}C_5&:=R(q)^3R(q^3)^4-\dfrac{q^6}{R(q)^3R(q^3)^4}=A(q^3)+(C_2-q^2)C_3\nonumber\\
    &=\dfrac{E_3^6}{E_{15}^6}+3q\dfrac{ E_5^5 E_3}{E_1 E_{15}^5}+2 q^3
    \intertext{\textup{and}}
   \label{C_6} C_6&:=R(q)^4R(q^3)^2+q^4\dfrac{R(q)^2}{R(q^3)^4}+\dfrac{R(q^3)^4}{R(q)^2}+\dfrac{q^4}{R(q)^4R(q^3)^2}=C_1C_2\notag\\
    &=2\dfrac{ E_3 E_5^5}{E_1 E_{15}^5}+5 q^2-18q^4\dfrac{ E_1 E_{15}^5 }{E_3 E_5^5}.
\end{align}
Using  Lemma \ref{LemmaC}, \eqref{C_4}, \eqref{C_5}, and  \eqref{C_6} in \eqref{colgen1eve1}, we arrive at \eqref{colorgen1}.

Next, we prove \eqref{colorgen2}. We have
\begin{align*}
\sum_{n=0}^{\infty}p_{[1^23^2]}(n)q^n=\dfrac{1}{E_1^2E_3^2}.
\end{align*}
Employing \eqref{1byE_1} in the above  and  then applying $[q^{5n+2}]$, we find that
\begin{align} \label{p1232}
&\sum_{n=0}^{\infty}p_{[1^23^2]}(5n+2)q^n\notag\\
&=5\dfrac{E_5^{10}E_{15}^{10}}{E_1^{12}E_3^{12}}\bigg(\bigg( R(q)^6R(q^3)^8+q^4\dfrac{R(q^3)^6}{R(q)^8}+q^8\dfrac{R(q)^8}{R(q)^6}+\dfrac{q^{12}}{R(q)^6R(q^3)^8}\bigg)\notag\\
&\quad+2q\bigg(R(q)^7R(q^3)^6+q^2\dfrac{R(q^3)^7}{R(q)^6}-q^8\dfrac{R(q)^6}{R(q^3)^7}-\dfrac{q^{10}}{R(q)^7R(q^3)^6}\bigg)\notag\\
&\quad+4q\bigg(R(q)R(q^3)^8-q^4\dfrac{R(q)^8}{R(q^3)}+q^6\dfrac{R(q^3)}{R(q)^8}-\dfrac{q^{10}}{R(q)R(q^3)^8}\bigg)\notag\\
&\quad+8q\bigg(R(q)^4R(q^3)^7-q^4\dfrac{R(q^3)^4}{R(q)^7}+q^6\dfrac{R(q)^7}{R(q^3)^4}-\dfrac{q^{10}}{R(q)^4R(q^3)^7}\bigg)\notag\\
&\quad+4q^2\bigg(R(q)^8R(q^3)^4+\dfrac{R(q^3)^8}{R(q)^4}+q^8\dfrac{R(q)^4}{R(q^3)^8}+\dfrac{q^{8}}{R(q)^8R(q^3)^4}\bigg)\notag\\
&\quad-8q^2\bigg(\dfrac{R(q^3)^7}{R(q)}-q^2R(q)^7R(q^3)-\dfrac{q^6}{R(q)^7R(q^3)}+q^8\dfrac{R(q)}{R(q^3)^7}\bigg)\notag\\
&\quad+20q^2\bigg(R(q)^5R(q^3)^5-q^2\dfrac{R(q^3)^5}{R(q)^5}-q^6\dfrac{R(q)^5}{R(q^3)^5}+\dfrac{q^{8}}{R(q)^5R(q^3)^5}\bigg)\notag\\
&\quad+27q^2\bigg(R(q)^2R(q^3)^6+q^4\dfrac{R(q)^6}{R(q^3)^2}+q^4\dfrac{R(q^3)^2}{R(q)^6}+ \dfrac{q^8}{R(q)^2R(q^3)^6}\bigg)\notag\\
&\quad+16q^3\bigg(R(q)^6R(q^3)^3-\dfrac{R(q^3)^6}{R(q)^3}+q^6\dfrac{R(q)^3}{R(q^3)^6}-\dfrac{q^{6}}{R(q)^6R(q^3)^3}\bigg)\notag\\
&\quad+30q^3\bigg(R(q^3)^5-\dfrac{q^6}{R(q^3)^5}+q^2R(q)^5-\dfrac{q^{4}}{R(q)^5}\bigg)\notag\\
&\quad+64q^3\bigg(R(q)^3R(q^3)^4+q^2\dfrac{R(q^3)^3}{R(q)^4}-q^4\dfrac{R(q)^4}{R(q^3)^3}-\dfrac{q^{6}}{R(q)^3R(q^3)^4}\bigg)\notag\\
&\quad+64q^4\bigg(R(q)R(q^3)^3-q^2\dfrac{R(q)^3}{R(q^3)}-q^2\dfrac{R(q^3)}{R(q)^3}+\dfrac{q^{4}}{R(q)R(q^3)^3}\bigg)\notag\\
&\quad+108q^4\bigg(R(q)^4R(q^3)^2+q^4\dfrac{R(q)^2}{R(q^3)^4}+\dfrac{R(q^3)^4}{R(q)^2}+\dfrac{q^4}{R(q)^4R(q^3)^2}\bigg)\notag\\
&\quad+108q^5\bigg(   R(q)^2 R(q^3)- \dfrac{R(q^3)^2}{R(q)}+q^2\dfrac{R(q)}{R(q^3)^2}-\dfrac{q^2}{R(q)^2 R(q^3)}\bigg)+45q^6\bigg).
\end{align}

By Lemma \ref{LemmaA} and Lemma \ref{LemmaC}, we find that
\begin{align}
    \label{C_7}C_7&:=\dfrac{R(q^3)^7}{R(q)}-q^2R(q)^7R(q^3)-\dfrac{q^6}{R(q)^7R(q^3)}+q^8\dfrac{R(q)}{R(q^3)^7}=C_4C_5,\\
    \label{C_8}C_8&:=R(q)R(q^3)^8-q^4\dfrac{R(q)^8}{R(q^3)}+q^6\dfrac{R(q^3)}{R(q)^8}-\dfrac{q^{10}}{R(q)R(q^3)^8}=A(q^3)C_2+q^4(C_3-AC_1),\\
    \label{C_9}C_9&:=R(q)^4R(q^3)^7-q^4\dfrac{R(q^3)^4}{R(q)^7}+q^6\dfrac{R(q)^7}{R(q^3)^4}-\dfrac{q^{10}}{R(q)^4R(q^3)^7}=C_2C_5-q^4(C_1C_4+C_3)\\ \intertext{and}
     \label{C_{10}}C_{10}&:=R(q)^7R(q^3)^6+q^2\dfrac{R(q^3)^7}{R(q)^6}-q^8\dfrac{R(q)^6}{R(q^3)^7}-\dfrac{q^{10}}{R(q)^7R(q^3)^6}\nonumber\\
     &=C_6(C_5-C_8+q^2C_4)-q^2A(q^3)+q^4(C_3-A).
\end{align}
Note that each of $C_7$--$C_{10}$ can be expressed in terms of $E_n$'s. Using \eqref{A}, Lemma \ref{LemmaC}, \eqref{C_4}, \eqref{C_5}, \eqref{C_6}, and \eqref{C_7}--\eqref{C_{10}} in \eqref{p1232}, we  arrive at \eqref{colorgen2}.

Finally, we sketch the proof of \eqref{colorgen3}. We have
\begin{align*}
\sum_{n=0}^{\infty}p_{[1^33^3]}(n)q^n=\dfrac{1}{E_1^3E_3^3}.
\end{align*}
As in the previous cases, we employ \eqref{1byE_1} in the above and  then apply $[q^{5n+3}]$ to deduce an identity similar to \eqref{p1232}.  The resulting identity can be shown to be equivalent to \eqref{colorgen3} with the help of Lemma \ref{LemmaC}, \eqref{A}, \eqref{C_1}--\eqref{C_3}, \eqref{C_4}, and \eqref{C_5}.
\end{proof}

\begin{proof}[Proof of Corollary \ref{coro2}] By the binomial theorem, we note that, for any  positive integer $k$,
\begin{align}
  \label{binomial-5}  E_k^5\equiv E_{5k} \;\textup{(mod\;5)}.
\end{align}
Therefore, from \eqref{colorgen1}, we have
\begin{align*}
\sum_{n=0}^{\infty} p_{[1^13^1]}(5n+1)q^n&\equiv \dfrac{E_5^4}{E_1E_{15}}+q^2\dfrac{E_{15}^4}{E_3E_5}~(\textup{mod}~5).
 \end{align*}
Employing \eqref{1byE_1} in the above and then applying $[q^{5n+4}]$, we find that
\begin{align*}
p_{[1^13^1]}(25n+21)\equiv0~(\textup{mod}~5),
\end{align*}
which is \eqref{colgen1R1}.

Now, we prove \eqref{colgen1R2}. Again from \eqref{colorgen1}, we have
\begin{align}\label{100}
  \notag\sum_{n=0}^{\infty}p_{[1^13^1]}(5n+1)q^n&\equiv\dfrac{E_5^5}{E_1^6E_{15}}+10q\dfrac{E_5^{9}}{E_1^2E_{15}}
  +q^2\dfrac{E_{15}^5}{E_3^6E_5}+45q^3\dfrac{E_5^4E_{15}^4}{E_1E_3}\\
  &\quad-90q^5\dfrac{E_{15}^{9}}{E_3^2E_5}~(\textup{mod}~25).
\end{align}
Next, we employ \eqref{1byE_1} followed by the extraction operator $\left[q^{5n+4}\right]$ and \eqref{binomial-5} to each term of the right side of the above to find the following identities.
\begin{align}\label{Doneinseparate1}
\notag\left[q^{5n+4}\right]\left\{\dfrac{E_5^5}{E_1^6E_{15}}\right\}&=\dfrac{5}{E_3}\bigg(63\dfrac{E_5^6}{E_1^7}
+52\times5^3q\dfrac{E_5^{12}}{E_1^{13}}+63\times5^5q^2\dfrac{E_5^{18}}{E_1^{19}}\\
\notag&\quad+6\times5^{8}q^3\dfrac{E_5^{24}}{E_1^{25}}+5^{10}q^4\dfrac{E_5^{30}}{E_1^{31}}\bigg)\\
&\equiv5\times63\dfrac{E_1^3E_5^4}{E_3}~(\textup{mod}~25),
\end{align}

\begin{align}\label{Doneinseparate2}
\notag\left[q^{5n+4}\right]\left\{10q\dfrac{E_5^9}{E_1^2E_{15}}\right\}&=\left[q^{5n+3}\right]\left\{10\dfrac{E_5^9}{E_1^2E_{15}}\right\}\\
\notag&=10\dfrac{E_1^9}{E_3}\cdot\dfrac{E_5^{10}}{E_1^{12}}\left(15q+10\left(R(q)^5-\dfrac{q^2}{R(q)^5}\right)\right)\\
&\equiv0~(\textup{mod}~25),
\end{align}

\begin{align}\label{Doneinseparate3}
\notag\left[q^{5n+4}\right]\left\{q^2\dfrac{E_{15}^5}{E_3^6E_5}\right\}&=\dfrac{5}{E_1}\bigg(63\dfrac{E_{15}^6}{E_3^7}
+52\times5^3q\dfrac{E_{15}^{12}}{E_3^{13}}+63\times5^5q^2\dfrac{E_{15}^{18}}{E_3^{19}}\\
\notag&\quad+6\times5^{8}q^3\dfrac{E_{15}^{24}}{E_3^{25}}+5^{10}q^4\dfrac{E_{15}^{30}}{E_3^{31}}\bigg)\\
\notag&\equiv5\times63q^2\dfrac{E_{15}^6}{E_1E_3^7}\\
&\equiv5\times63q^2\dfrac{E_3^3E_{15}^4}{E_1}~(\textup{mod}~25),
\end{align}
\begin{align}\label{Doneinseparate4}
\notag\left[q^{5n+4}\right]\left\{45q^3\dfrac{E_5^4E_{15}^4}{E_1E_3}\right\}&=\left[q^{5n+1}\right]\left\{45\dfrac{E_5^4E_{15}^4}{E_1E_3}\right\}\\
\notag&\equiv45E_1^4E_3^4\left(\dfrac{E_5^4}{E_1E_{15}}+q^2\dfrac{E_{15}^4}{E_3E_5}\right)\\
\notag&\equiv45\left(\dfrac{E_1^3E_5^4E_3^4}{E_{15}}+q^2\dfrac{E_1^4E_3^3E_{15}^4}{E_5}\right)\\
&\equiv45\left(\dfrac{E_1^3E_{5}^4}{E_3}+q^2\dfrac{E_3^3E_{15}^4}{E_1}\right)~(\textup{mod}~25)
\end{align}
and
\begin{align}\label{Doneinseparate5}
\notag\left[q^{5n+4}\right]\left\{90q^5\dfrac{E_{15}^9}{E_3^2E_{5}}\right\}&=90\dfrac{E_{25}^{10}}{E_1E_3^3}\left(15q^5+10q^2
\left(R(q^3)^5-\dfrac{q^6}{R(q^3)^5}\right)\right)\\
&\equiv0~(\textup{mod}~25).
\end{align}
Using \eqref{Doneinseparate1}--\eqref{Doneinseparate5} in \eqref{100}, we arrive at
\begin{align*}
\sum_{n=0}^{\infty}p_{[1^13^1]}(25n+21)q^n&\equiv10\left(\dfrac{E_1^3E_{5}^4}{E_3}+q^2\dfrac{E_3^3E_{15}^4}{E_1}\right)~(\textup{mod}~25).
\end{align*}
Employing Lemma \ref{Dissect} in the above and then applying $\left[q^{5n}\right]$, we obtain
\begin{align*}
&\sum_{n=0}^{\infty}p_{[1^13^1]}(125n+21)q^n\notag\\
&\equiv10\bigg(\dfrac{E_1^4E_5^3E_{15}^5}{E_3^6}\bigg(R(q)^3R(q^3)^4-\dfrac{q^6}{R(q)^3R(q^3)^4}+q^3\bigg(25
-3\bigg(\dfrac{R(q)^3}{R(q^3)}+\dfrac{R(q^3)}{R(q)^3}\bigg)\bigg)\bigg)\notag\\
&\quad -q\dfrac{E_3^4E_5^5E_{15}^3}{E_1^6}\bigg(3\bigg(R(q)^4R(q^3)^2+q^4\dfrac{R(q)^2}{R(q^3)^4}+\dfrac{R(q^3)^4}{R(q)^2}+\dfrac{q^4}{R(q)^4R(q^3)^2}\bigg)\notag\\
&\quad -3\bigg(R(q) R(q^3)^3+\dfrac{q^4}{R(q) R(q^3)^3}\bigg)+q\bigg(\dfrac{R(q^3)^3}{R(q)^4}-q^2\dfrac{R(q)^4}{R(q^3)^3}\bigg)\notag\\
&\quad+3q\bigg(R(q)^2 R(q^3)- \dfrac{R(q^3)^2}{R(q)}+q^2\dfrac{R(q)}{R(q^3)^2}-\dfrac{q^2}{R(q)^2 R(q^3)}\bigg)\bigg)\bigg)\;(\textup{mod}\; 25).
\end{align*}
Applying Lemma \ref{LemmaC}, \eqref{C_4}, \eqref{C_5}, and \eqref{C_6} in the above, we find that
\begin{align*}
    \sum_{n=0}^{\infty}p_{[1^13^1]}(125n+21)q^n&\equiv10\left(\dfrac{E_{5}^4}{E_1E_{15}}+q^2\dfrac{E_{15}^4}{E_3E_5}\right)~(\textup{mod}~25).
\end{align*}
Once again employing \eqref{1byE_1} in the above and then applying  $\left[q^{5n+4}\right]$, we arrive at \eqref{colgen1R2}.
\end{proof}

\begin{proof}[Proof of Corollary \ref{coro3}]
By \eqref{colorgen2}, it is obvious that \eqref{Colgen2R1} is true. It remains to prove \eqref{Colgen2R2}.
Using \eqref{binomial-5} in \eqref{colorgen2}, we find that
\begin{align}
  \label{start}\sum_{n=0}^{\infty}p_{[1^23^2]}(5n+2)q^n&\equiv5\dfrac{E_5^8}{E_1^2E_{15}^2}+10q^2\dfrac{E_5^{3}E_{15}^3}{E_1E_{3}}
  +5q^4\dfrac{E_{15}^8}{E_3^2E_5^2}~(\textup{mod}~25).
\end{align}

From \eqref{Doneinseparate2} and \eqref{Doneinseparate5}, we have
\begin{align*}
   [q^{5n+3}]\bigg\{\dfrac{1}{E_1^2}\bigg\}&\equiv0\:\textup{(mod\:5)}\\\intertext{and}
  [q^{5n+4}]\bigg\{\dfrac{1}{E_3^2}\bigg\}&\equiv0\:\textup{(mod\:5)}.
\end{align*}

Employing \eqref{1byE_1} in \eqref{start}, applying $[q^{5n+3}]$ and then with an aid from the above congruences, we obtain
\begin{align*}
&\sum_{n=0}^{\infty}p_{[1^23^2]}(25n+17)q^n&\notag\\&\equiv 10\dfrac{E_5^{5}E_{15}^5}{E_1^3E_3^3}\bigg(\bigg(R(q)^3R(q^3)^4+q^2\dfrac{R(q^3)^3}{R(q)^4}-q^4\dfrac{R(q)^4}{R(q^3)^3}-\dfrac{q^{6}}{R(q)^3R(q^3)^4}\bigg) \notag\\
&\quad+2q\bigg(R(q)^4R(q^3)^2+q^4\dfrac{R(q)^2}{R(q^3)^4}+\dfrac{R(q^3)^4}{R(q)^2}+\dfrac{q^4}{R(q)^4R(q^3)^2}\bigg)\notag\\
&\quad+3q\bigg(R(q)R(q^3)^3-q^2\dfrac{R(q)^3}{R(q^3)}-q^2\dfrac{R(q^3)}{R(q)^3}+\dfrac{q^{4}}{R(q)R(q^3)^3}\bigg)\notag\\
&\quad+6q^2\bigg(R(q)^2 R(q^3)- \dfrac{R(q^3)^2}{R(q)}+q^2\dfrac{R(q)}{R(q^3)^2}-\dfrac{q^2}{R(q)^2 R(q^3)}\bigg)\bigg)\textup{(mod\:25)},
\end{align*} which by Lemma \ref{LemmaC}, \eqref{C_4}, \eqref{C_5} , and  \eqref{C_6} reduces to
\begin{align}
   \label{Reach} \sum_{n=0}^{\infty}p_{[1^23^2]}(25n+17)q^n&\equiv10\bigg(E_1^2E_3^3\dfrac{E_5^4}{E_{15}}+q^2E_1^3E_3^2\dfrac{E_{15}^4}{E_{5}}\bigg)\;\textup{(mod\:25)}.
\end{align}
We again employ \eqref{E_1} in the above and then apply $[q^{5n+1}]$ and Lemma \ref{LemmaC}, to obtain
\begin{align*}
    \sum_{n=0}^{\infty}p_{[1^23^2]}(125n+42)q^n&\equiv20\bigg(4\dfrac{E_5^8}{E_1^2E_{15}^2}+3q^2\dfrac{E_5^3E_{15}^3}{E_1E_3}
    +4q^4\dfrac{E_{15}^8}{E_3^2E_5^2}\bigg)\;\textup{(mod\:25)}.
\end{align*}
The procedure from \eqref{start} to \eqref{Reach} can be repeated in the above, to arrive at
\begin{align}
    \label{Final}\sum_{n=0}^{\infty}p_{[1^23^2]}(625n+417)q^n&\equiv10\bigg(E_1^2E_3^3\dfrac{E_5^4}{E_{15}}+q^2E_1^3E_3^2\dfrac{E_{15}^4}{E_{5}}\bigg)\;\textup{(mod\:25)}.
\end{align}
From \eqref{Reach} and \eqref{Final}, it follows that
\begin{align*}
   p_{[1^23^2]}(625n+417)\equiv p_{[1^23^2]}(25n+17)\;\textup{(mod\:25)},
\end{align*}
 which by induction gives \eqref{Colgen2R2}.
\end{proof}

\begin{proof}[Proof of \eqref{Colgen3R1}]

Using \eqref{E_1}, \eqref{C_3}, and \eqref{binomial-5}, we can easily see that
\begin{align*}
    \sum_{n=0}^{\infty}p_{[1^33^3]}(n)q^n&=\dfrac{1}{E_1^3E_3^3}\equiv\dfrac{E_1^2E_3^2}{E_5E_{15}}\;\;\textup{(mod\:\:5)}\\\intertext{and}
    [q^{5n+3}]\{E_1^2E_3^2\}&\equiv0\;\;\textup{(mod\:\:5)}.
\end{align*}
Congruence \eqref{Colgen3R1} follows immediately from the above.
\end{proof}

\begin{proof}[Proof of Theorem \ref{TheoCongruence}]
We recall from Section \ref{intro} that
\begin{align*}
  \sum_{n=0}^\infty p_{\beta}(3n+1)q^n&=\dfrac{E_3^3}{E_1^2}.
  \end{align*}
Employing Lemma \ref{Dissect} in the above and then applying $[q^{5n+2}]$, we find that
\begin{align}\label{beta-eq}
 & \sum_{n=0}^{\infty}p_{\beta}(15n+7)q^n\notag\\
 &=5\dfrac{E_5^{10}E_{15}^3}{E_1^{12}}\bigg(4qC_2+10q^2A+4q^2C_4-3q^3C_1^2+21q^3+R(q)^6R(q^3)^3 -\dfrac{q^6}{R(q)^6R(q^3)^3}\notag\\
 &\quad-12q\bigg(R(q)^4R(q^3)^2
  +\dfrac{q^4}{R(q)^4R(q^3)^2}\bigg)\bigg)
  +12q^2\bigg(\dfrac{R(q^3)^2}{R(q)}-q^2\dfrac{R(q)}{R(q^3)^2}\bigg)\bigg).
\end{align}
Congruence \eqref{15n7} follows immediately from the above.

Now, we prove the remaining two congruences. From Section \ref{intro}, we also recall that
\begin{align*}
  \sum_{n=0}^\infty p_{\beta}(9n+5)q^n&=3\dfrac{E_3^6}{E_1^5}.
  \end{align*}
Employing \eqref{binomial-5} in the above, we have
\begin{align*}
  \sum_{n=0}^\infty p_{\beta}(9n+5)q^n&\equiv3\dfrac{E_{15}E_3}{E_5}~\textup{(mod~5)},
  \end{align*}
  which, by \eqref{E_1}, can be rewritten as
  \begin{align*}
  \sum_{n=0}^\infty p_{\beta}(9n+5)q^n&\equiv3\dfrac{E_{15}E_{75}}{E_5}\left(R(q^{15})-q^3+\dfrac{q^6}{R(q^{15})}\right)~\textup{(mod~5)}.
  \end{align*}
Applying $[q^{5n+2}]$ and $[q^{5n+4}]$ in the above, we obtain \eqref{45n23} and \eqref{45n41}, respectively.
\end{proof}

\begin{remark}
 We could not effectively transform $\big(R(q^3)^2/R(q)-q^2R(q)/R(q^3)^2\big)$ in \eqref{beta-eq} into an expression involving only  $E_n$'s, which could have helped in reducing the right side of \eqref{beta-eq} in terms of $E_n$'s with an aid from \eqref{C_3}. The transformed equivalent form of \eqref{beta-eq}  might have lead to congruences modulo higher  powers of 5, including the following congruences
conjectured by Zhang and Shi \cite[Conjecture 6]{zhang}.
\begin{align*}
p_{\beta}(3\cdot 5^2n+22)&\equiv p_{\beta}(3\cdot 5^2n+52)\equiv p_{\beta}(3\cdot 5^2n+67)\equiv0~(\textup{mod}~5^2),\\
p_{\beta}(3\cdot 5^4n+547)&\equiv p_{\beta}(3\cdot 5^4n+1297)\equiv p_{\beta}(3\cdot 5^4n+1672)\equiv0~(\textup{mod}~5^3).
\end{align*}
\end{remark}

\section{\textbf{Proofs of Theorems \ref{Theo3}--\ref{Theo4}  using identities for $R(q)$ and $R(q^4)$}}\label{SectionR1R4}
\begin{proof}[Proof of Theorem \ref{Theo3}]
We have
\begin{align*}
\sum_{n=0}^\infty b_4(n)q^n=\dfrac{E_4}{E_1}.
\end{align*}
Employing Lemma \ref{Dissect}  in the above and then applying $[q^{5n+3}]$, we find that
\begin{align}
\label{b}\sum_{n=0}^\infty b_4(5n+3)q^n&=\dfrac{E_5^5E_{20}}{E_1^6}\bigg(3\bigg(R(q)R(q^4)+\dfrac{q^2}{R(q)R(q^4)}\bigg)\notag\\
&\quad-q\bigg(5+\bigg(\dfrac{R(q)^4}{R(q^4)}-\dfrac{R(q^4)}{R(q)^4}\bigg)\bigg)\bigg).
\end{align}

Now, from Lemma \ref{LemmaA} and Lemma \ref{LemmaF}, it follows that
\begin{align}
\label{D_2}D_2&:=\dfrac{R(q)^4}{R(q^4)}-\dfrac{R(q^4)}{R(q)^4}=BF+B(q^2)\notag\\
&=8q\dfrac{ E_1 E_{10}^5}{E_2 E_5^5}+4q^2\dfrac{E_2 E_{20}^5 }{E_4 E_{10}^5}+16q^3\dfrac{E_1 E_{20}^5 }{E_4 E_5^5}.
\end{align}
Using \eqref{D_1} and \eqref{D_2} in \eqref{b}, we find that
\begin{align}
\label{b4gen}&\sum_{n=0}^\infty b_4(5n+3)q^n\notag \\
&=3\dfrac{E_4E_{10}^{10}}{E_1^5E_2^2E_{20}^4}+q\dfrac{E_5^5E_{20}}{E_1^6}-
8q^2\dfrac{E_{10}^5E_{20}}{E_1^5E_2}-4q^3\dfrac{E_2E_5^5E_{20}^6}{E_1^6
E_4E_{10}^5}-16q^4\dfrac{E_{20}^6}{E_1^5E_4} \notag \\
&=\left(\dfrac{E_{10}^5}{E_2^4E_{20}^3}-4q^2\dfrac{E_{20}^2}{E_2^3E_4}\right)
\left(3\dfrac{E_2^2E_4E_{10}^5}{E_1^5E_{20}}+q\dfrac{E_2^4E_5^5E_{20}^4}{E_1^6E_{10}^5}
+4q^2\dfrac{E_2^3E_{20}^4}{E_1^5}\right).
\end{align}

From \cite[Lemma 2.2.2]{bb-distinct}, we recall that
\begin{align*}
   \dfrac{E_5^5}{E_1^4E_{10}^3}&=\dfrac{E_5}{E_2^2E_{10}}+4q\dfrac{E_{10}^2}{E_1^3E_{2}}.
\end{align*}
Replacing $q$ by $q^2$ in the above and then using the resulting identity in \eqref{b4gen}, we arrive at \eqref{ell1}.
\end{proof}
\begin{proof}[Proof of Theorem \ref{Theo4}]
We have
\begin{align*}
\sum_{n=0}^{\infty}a_{4}(n)q^{n}=\dfrac{E_{4}^{4}}{E_1}.
\end{align*}
Employing  Lemma \ref{Dissect}  in the above and then applying $[q^{5n}]$, we find that
\begin{align}
\label{c}&\sum_{n=0}^{\infty}a_{4}(5n)q^{n}\notag\\
\notag&=\dfrac{E_5^5E_{20}^4}{E_1^6}\bigg(R(q)^4R(q^4)^4+\dfrac{q^8}{R(q)^4R(q^4)^4}+q\bigg(-4\bigg(R(q)^3R(q^4)^3+\dfrac{q^6}{R(q)^3R(q^4)^3}\bigg)\\
\notag&\quad-3\bigg(\dfrac{R(q^4)^4}{R(q)}-q^6\dfrac{R(q)}{R(q^4)^4}\bigg)\bigg)+q^2\bigg(4\bigg(R(q)^2R(q^4)^2+\dfrac{q^4}{R(q)^2R(q^4)^2}\bigg)\\
\notag&\quad-8\bigg(\dfrac{R(q^4)^3}{R(q)^2}-\dfrac{R(q)^2}{R(q^4)^3}\bigg)\bigg)+q^3\bigg(24\bigg(R(q)R(q^4)+\dfrac{q^2}{R(q)R(q^4)}\bigg)\\
&\quad-2\bigg(\dfrac{R(q^4)^2}{R(q)^3}-q^2\dfrac{R(q)^3}{R(q^4)^2}\bigg)\bigg)+q^4\bigg(-25-8\bigg(\dfrac{R(q)^4}{R(q^4)}-\dfrac{R(q^4)}{R(q)^4}\bigg)\bigg)\bigg).
\end{align}

It follows from Lemma \ref{LemmaA} and Lemma \ref{LemmaD} that
\begin{align}
 \label{D_3}&D_3:=\dfrac{R(q^4)^2}{R(q)^3}-q^2\dfrac{R(q)^3}{R(q^4)^2}=A-D_1D_2,\\
 \label{D_4}&D_4:=\dfrac{R(q^4)^3}{R(q)^2}-q^4\dfrac{R(q)^2}{R(q^4)^3}=D_1D_3+q^2D_2,\\
 \label{D_5}&D_5:=\dfrac{R(q^4)^4}{R(q)}-q^6\dfrac{R(q)}{R(q^4)^4}=D_1D_4-q^2D_3,\\
 \label{D_6}&D_6:=R(q)^2R(q^4)^2+\dfrac{q^4}{R(q)^2R(q^4)^2}=D_1^2-2q^2,\\
 \label{D_7}&D_7:=R(q)^3R(q^4)^3+\dfrac{q^6}{R(q)^3R(q^4)^3}=D_1^3-3q^2D_1\\\intertext{and}
 \label{D_8}&D_8:=R(q)^4R(q^4)^4+\dfrac{q^8}{R(q)^4R(q^4)^4}=D_6^2-2q^4.
\end{align}
Using \eqref{D_1}, \eqref{D_2}, and \eqref{D_3}--\eqref{D_8} in \eqref{c}, we
arrive at \eqref{4coregen}.
\end{proof}

\section*{\textbf{Acknowledgement}}
 The first author was partially supported by Grant no. MTR/2018/000157 of Science \& Engineering Research Board (SERB), DST, Government of India under the MATRICS scheme. The third author was partially supported by Council of Scientific \& Industrial Research (CSIR), Government of India under CSIR-JRF scheme. The authors thank both the funding agencies.

\end{document}